\documentclass[11pt,reqno]{amsart}
\usepackage[T1]{fontenc}
\usepackage[utf8]{inputenc}
\usepackage{lmodern}
\usepackage{amsmath,amssymb,mathtools}
\usepackage{microtype}
\usepackage[margin=1.12in]{geometry}
\usepackage{xcolor}
\usepackage[colorlinks=true,linkcolor=blue!45!black,citecolor=blue!45!black,urlcolor=blue!45!black]{hyperref}
\hypersetup{pdftitle={Minimal hypersurfaces of finite delta-index in R^4 and R^5},
            pdfauthor={Marcos Ranieri}}
\newtheorem{theorem}{Theorem}[section]
\newtheorem{proposition}[theorem]{Proposition}
\newtheorem{lemma}[theorem]{Lemma}
\newtheorem{corollary}[theorem]{Corollary}
\numberwithin{equation}{section}
\newcommand{\R}{\mathbb R}
\newcommand{\Sph}{\mathbb S}
\newcommand{\dd}{\,d\mu}
\newcommand{\td}{\,d\widetilde\mu}
\newcommand{\ind}{\operatorname{ind}}
\newcommand{\Vol}{\operatorname{Vol}}
\newcommand{\Ric}{\operatorname{Ric}}
\newcommand{\Scal}{\operatorname{Scal}}
\newcommand{\biRic}{\operatorname{biRic}}
\newcommand{\supp}{\operatorname{supp}}
\newcommand{\ess}{\mathrm{ess}}
\newcommand{\tg}{\widetilde g}
\newcommand{\tgrad}{\widetilde\nabla}
\newcommand{\tDelta}{\widetilde\Delta}
\newcommand{\Cc}{C_c^\infty}
\DeclareMathOperator{\diver}{div}
\newcommand{\dist}{\operatorname{dist}}
\allowdisplaybreaks[1]
\title[Minimal hypersurfaces of finite $\delta$-index]
      {Minimal hypersurfaces of finite $\delta$-index in
       $\mathbb R^4$ and $\mathbb R^5$}
\author{Marcos Ranieri}
\address{Instituto de Matem\'atica, Universidade Federal de Alagoas, Macei\'o, Brazil}
\email{marcos.ranieri@im.ufal.br}
\date{September 12, 2026}
\subjclass[2020]{Primary 53A10; Secondary 53C21, 58J50}
\keywords{Minimal hypersurface, Bernstein theorem, delta-stability, Morse index, total curvature, mu-bubble}

\begin{document}
\begin{abstract}
Let $X:M^n\to\mathbb R^{n+1}$ be a complete, connected, two-sided
minimal immersion without boundary, where $n=3,4$. We prove that
finite $\delta$-index, finite Morse index, and finite total curvature
are equivalent for $\delta>((n-1)/n)^2$, without assumptions on
properness, volume growth, or topology. The main estimate shows
that $\delta$-stability outside a compact set and finite-dimensional
$H_c^1(M;\mathbb R)$ imply intrinsic Euclidean volume growth for
$\delta>(n-2)/n$. We also deduce the sharp $\delta$-stable Bernstein
theorem in this latter range from results of Hong--Li--Wang
and Florit-Simon.
\end{abstract}
\maketitle

\section{Introduction}

Let $X:M^n\to\R^{n+1}$ be a complete, connected, two-sided minimal immersion without boundary, where $n\geq3$. For $\delta>0$, set
\begin{equation}\label{eq:form}
 Q_\delta(f)=\int_M\bigl(|\nabla f|^2-\delta|A|^2f^2\bigr)\dd,
 \qquad f\in\Cc(M),
\end{equation}
where $A$ is the second fundamental form. We call $X$
\emph{$\delta$-stable} if $Q_\delta\geq0$. The \emph{$\delta$-index},
denoted by $\ind_\delta(M)$, is the supremum of the dimensions of
subspaces of $\Cc(M)$ on which $Q_\delta$ is negative definite.
Thus $\ind_1$ is the Morse index, and $\delta$-stability is weaker
than stability when $0<\delta<1$.

The stable Bernstein theorem was proved for $n=3$ by Chodosh--Li
\cite{CL} and for $n=4$ by Chodosh--Li--Minter--Stryker \cite{CLMS}.
Mazet \cite{Mazet} proved the case $n=5$.
For other proofs in dimension $n=3$, see also
\cite{CCMMR,CMR,CLA,HW}.
In dimensions $n=3,4$, finite Morse index is also equivalent to
finite total curvature, without a volume growth assumption
\cite[Theorem~5]{CL}, \cite[Theorem~1.4]{CLMS}.

The $n$-dimensional catenoid is $(n-2)/n$-stable \cite{TZ}, so Bernstein rigidity cannot hold at or below this coefficient \footnote{%
Here we translate Tam--Zhou's convention: their parameter
$\delta_{\mathrm{TZ}}$ appears through the potential coefficient
$1-\delta_{\mathrm{TZ}}$ in \cite[Equation~(1.2)]{TZ}. Thus their
$2/n$-stability corresponds to $(n-2)/n$-stability in the convention
of \eqref{eq:form}}. Hong--Li--Wang \cite[Theorem~1.7]{HLW} proved rigidity for $\delta>3/8$ when $n=3$ and $\delta>2/3$ when $n=4$. Under extrinsic Euclidean volume growth, their result reaches $\delta>(n-2)/n$ in both dimensions \cite[Corollary~1.4]{HLW}. They also proved intrinsic Euclidean volume growth in this range for simply connected hypersurfaces with finitely many ends \cite[Proposition~1.10(b)]{HLW}.
Intrinsic volume estimates at the same thresholds are also stated
by Cheng--Wei \cite[Theorem~4.2]{CW}. In dimension $n=3$, Catino--Mari--Mastrolia--Roncoroni \cite[Theorems~1.11 and~1.13]{CMMR} proved rigidity for $\delta>1/3$ under properness and showed that a $1/3$-stable hypersurface has one end or is a catenoid.

Florit-Simon \cite{FS} proved that finite intrinsic Euclidean
volume growth of a complete minimal immersion implies properness
and extrinsic Euclidean volume growth. We combine this comparison
with the results of Hong--Li--Wang and the end theorem of
Cheng--Zhou \cite{CZ} to reach the catenoid threshold for $n=3,4$
without additional hypotheses. The reduction also yields the local
curvature estimate needed below.

\begin{theorem}\label{thm:bernstein}
Let $X:M^n\to\R^{n+1}$ be a complete, connected, two-sided minimal
immersion without boundary, where $n\in\{3,4\}$. If
\[
 \delta>\frac{n-2}{n}
 \qquad\text{and}\qquad Q_\delta\geq0,
\]
then $X$ is an isometry onto an affine hyperplane.
\end{theorem}

The thresholds are $\delta>1/3$ for $n=3$ and $\delta>1/2$ for $n=4$; the catenoid shows that both are sharp. A point-picking argument removes the curvature bound used in the reduction.

Our main result concerns finite $\delta$-index.

\begin{theorem}\label{thm:index}
Let $X:M^n\to\R^{n+1}$ be a complete, connected, two-sided minimal
immersion without boundary, where $n\in\{3,4\}$, and let
\[
 \delta>\left(\frac{n-1}{n}\right)^2.
\]
Then the following conditions are equivalent:
\begin{enumerate}\label{eq:equivalence}
 \item[i)]$\ind_\delta(M)<\infty$;
 \item[ii)]$ \ind_1(M)<\infty;$
\item[iii)]$\int_M|A|^n\dd<\infty.$
\end{enumerate}
Under these conditions $X$ is proper,
\begin{equation}\label{eq:extrinsic-volume}
 \Vol\bigl(X^{-1}(B_R^{\R^{n+1}}(0))\bigr)\leq CR^n
 \qquad(R\geq1)
\end{equation}
for some $C<\infty$, and
\begin{equation}\label{eq:curvature-decay}
 |X(x)|\,|A|(x)\longrightarrow0
 \qquad\text{as }x\longrightarrow\infty.
\end{equation}
Moreover, $\ind_\tau(M)<\infty$ for every $\tau>0$.
\end{theorem}

The ranges are $\delta>4/9$ for $n=3$ and $\delta>9/16$ for $n=4$.
The implication from finite total curvature to finite index is
classical \cite{Tysk}; the ordinary-index converse is due to
\cite[Theorem~5]{CL} and \cite[Theorem~1.4]{CLMS}.
The new implication is therefore
$\ind_\delta<\infty\Rightarrow\ind_1<\infty$ for the stated
range below one. Thus the weaker index defines the same class of
complete minimal hypersurfaces, without assuming properness,
a volume growth bound, or finite topology.

The geometric step is the Proposition~\ref{prop:volume}. For $n=3,4$
and $\delta>(n-2)/n$, stability outside a compact set and
$\dim H_c^1(M;\R)<\infty$ imply intrinsic Euclidean volume growth.
Here $H_c^1(M;\R)$ denotes cohomology with compact support.
We construct weighted $\mu$-bubble separators in compact bands
contained in the stable exterior. The estimates of
\cite[Section~5]{HLW} bound the conformal area of each component.
After filling relatively compact complementary components,
Poincar\'e duality bounds their number by $1+\dim H_c^1(M;\R)$.
The resulting bound for the total boundary area gives the
volume estimate by the Michael--Simon inequality.

For ordinary stability, the exterior construction and the control
of separator components by harmonic one-forms already appear in
\cite[Section~6]{CL} and in the proof of
\cite[Theorem~1.4]{CLMS}, using \cite{LW02}.
We combine this construction with the $\delta$-stability estimates
of Hong--Li--Wang; finite-dimensional compactly supported
cohomology replaces the simple connectivity in their global
volume argument.

For $n\geq3$, Fu \cite{Fu} proved that global $\delta$-stability forces the
vanishing of square-integrable harmonic one-forms and of
$H_c^1(M;\R)$ when $\delta>((n-1)/n)^2$.
We localize the Bochner--Kato estimate to obtain finite
dimensionality under exterior stability. For $n=3,4$, this supplies
the cohomological hypothesis of Proposition~\ref{prop:volume}.
Florit-Simon's theorem then gives properness and
extrinsic volume growth. The local curvature estimate and the cone calculation of
Hong--Li--Wang imply $|X||A|\to0$. The extrinsic Hardy inequality
then gives finite Morse index.

The spectrum is a consequence of this geometry. We use
$\Delta=\diver\nabla$, so that $-\Delta$ is nonnegative.

\begin{corollary}\label{cor:spectrum}
Let $X:M^n\to\R^{n+1}$ be a complete, connected, two-sided minimal
immersion without boundary, where $n\in\{3,4\}$. Suppose that
$\ind_\delta(M)<\infty$ for some
$\delta>((n-1)/n)^2$. Then the self-adjoint Schr\"odinger operators
\[
 L_\tau=-\Delta-\tau|A|^2,\qquad \tau\geq0,
\]
on $L^2(M)$ satisfy $\sigma_{\ess}(L_\tau)=[0,\infty)$.
For each $\tau>0$, the negative spectrum consists of finitely
many eigenvalues, counted with multiplicity.
\end{corollary}

The corollary follows from Lu--Zhou's theorem \cite{LZ}
and relative compactness of the curvature potential.
The Laplacian conclusion under finite total curvature
was also recorded in \cite[Remark~4.1]{DLZ}.

All metrics and volumes are induced on the abstract manifold;
extrinsic volumes count the multiplicity of the immersion.
Balls $B_R(p)$ without an ambient superscript are intrinsic.
The notation $x\to\infty$ means that $x$ leaves every compact subset
of $M$. A complete minimal immersion in Euclidean space without
boundary is noncompact, since its coordinate functions are harmonic.

Section~2 proves Theorem~\ref{thm:bernstein} and the local
curvature estimate. In Section~3, we establish the cohomological
and exterior volume estimates, derive curvature decay, and prove
Theorem~\ref{thm:index} and Corollary~\ref{cor:spectrum}.

\section{Sharp Bernstein rigidity and a local curvature estimate}

We first record that the stability inequality lifts to an arbitrary
Riemannian covering. If $\pi:\widehat M\to M$ is a covering and
$f\in\Cc(\widehat M)$, set
\[
 F(p)=\left(\sum_{q\in\pi^{-1}(p)}f(q)^2\right)^{1/2}.
\]
The function $F$ is compactly supported and locally Lipschitz.
Cauchy--Schwarz gives, almost everywhere,
\[
 |\nabla F|^2\leq
 \sum_{q\in\pi^{-1}(p)}|\nabla f(q)|^2.
\]
The potential integrals and the squared $L^2$ norms agree on the
two spaces. Approximation by smooth compactly supported functions
therefore shows that
\begin{equation}\label{eq:cover}
 Q_{\delta,\widehat M}(f)\geq Q_{\delta,M}(F).
\end{equation}
In particular, $\delta$-stability lifts to the covering.

\begin{proof}[Proof of Theorem~\ref{thm:bernstein}]
By decreasing the coefficient if necessary, it is enough to consider
$(n-2)/n<\delta<1$. First suppose that $|A|$ is bounded.
Lift the immersion to the universal cover $\widehat M$.
It is complete, two-sided, and $\delta$-stable by
\eqref{eq:cover}. In particular it is $(n-2)/n$-stable.
Theorem~1.1 and Corollary~1.1 of Cheng--Zhou \cite{CZ} imply
that $\widehat M$ has one end or is a catenoid. Their curvature
growth hypotheses are satisfied by bounded curvature: they require
$o(\log R)$ in dimension three and
$o(R^{(n-3)/2})$ in dimension at least four.
In either case $\widehat M$ has finitely many ends.

Since $\widehat M$ is simply connected, Proposition~1.10(b) of
Hong--Li--Wang \cite{HLW} gives
\[
 \Vol(B_R^{\widehat M}(\widehat p))\leq CR^n
 \qquad(R\geq1)
\]
for a finite constant $C$. Florit-Simon \cite[Theorem~1.1]{FS}
implies that the lifted immersion is proper and has extrinsic
Euclidean volume growth. In dimensions three and four,
\[
 \max\left\{\frac{n-2}{n},
             \frac{(n-2)^2}{4(n-1)}\right\}
 =\frac{n-2}{n}.
\]
Corollary~1.4 of \cite{HLW} now gives
$\widehat A\equiv0$, and hence $A\equiv0$ on $M$.

We remove the curvature bound by point-picking. If $M$ were
nonflat, choose $p\in M$ with $|A|(p)>0$ and let $R_j\to\infty$.
The continuous function
\[
 x\longmapsto |A|(x)\bigl(R_j-d_M(p,x)\bigr)
\]
attains its maximum on $\overline{B_{R_j}(p)}$. At a maximizing
point $q_j$, put
$b_j=|A|(q_j)$ and $s_j=R_j-d_M(p,q_j)$. Then
\begin{equation}\label{eq:point-picking}
 b_js_j\geq |A|(p)R_j\longrightarrow\infty,
 \qquad
 |A|\leq2b_j\quad\hbox{on }B_{s_j/2}(q_j).
\end{equation}
Dilate the immersion by $b_j$ about $X(q_j)$.
The resulting pointed domains have radii tending to infinity,
$|A_j|(q_j)=1$, and $|A_j|\leq2$.

The rescaled immersions have base point at the origin,
uniformly bounded second fundamental form, and intrinsic distance
from the base point to the boundary tending to infinity.
Proposition~8.2 of \cite{ChNotes} therefore gives a complete
pointed limit, with smooth convergence on compact subsets of
the abstract limit manifold. In these pointed parametrizations
the preimages of $q_j$ converge to $q_\infty$, so the normalization
$|A_j|(q_j)=1$ survives at $q_\infty$.

The unit normals converge after passing to a subsequence.
Every compactly supported test function on the limit can be
transported to the approximating domains, so the stability
inequality passes to the limit. Thus $M_\infty$ is complete,
two-sided, and $\delta$-stable, with $|A_\infty|\leq2$ and
$|A_\infty|(q_\infty)=1$, contradicting the bounded-curvature case.

Finally, $A\equiv0$ places the image in an affine hyperplane.
The immersion into that hyperplane is a local isometry from
a complete connected manifold. It is therefore a covering,
and the hyperplane is simply connected. Thus the immersion
is an isometry onto the hyperplane.
\end{proof}

\begin{proposition}\label{prop:curvature}
Fix $n\in\{3,4\}$ and $\delta>(n-2)/n$. There is a constant
$C=C(n,\delta)$ such that, if $M^n\to\R^{n+1}$ is a complete
two-sided minimal immersion without boundary and $\Omega\subset M$
is a $\delta$-stable open set, then
\begin{equation}\label{eq:local-curvature}
 |A|(p)\leq\frac{C}{\dist_M(p,M\setminus\Omega)},
 \qquad p\in\Omega.
\end{equation}
The right-hand side is zero when $\Omega=M$.
In particular, if $M$ is $\delta$-stable outside a nonempty
compact domain $K$, then
\begin{equation}\label{eq:exterior-curvature}
 |A|(x)\leq\frac{C}{d_M(x,K)}
 \qquad(x\in M\setminus K).
\end{equation}
\end{proposition}

\begin{proof}
Otherwise, there are examples $(M_j,\Omega_j,p_j)$ with
$a_j=|A_j|(p_j)>0$ and
\[
 a_j\dist_{M_j}(p_j,M_j\setminus\Omega_j)\longrightarrow\infty.
\]
Choose finite radii $r_j$ below these distances with $a_jr_j\to\infty$,
and apply the point-picking construction on $B_{r_j}(p_j)$.
After dilation, the pointed domains have radii tending to infinity,
$|A_j|\leq2$, and $|A_j|(q_j)=1$.
The preceding compactness argument gives a complete, two-sided,
$\delta$-stable nonflat limit, contrary to
Theorem~\ref{thm:bernstein}. Taking $\Omega=M\setminus K$
gives \eqref{eq:exterior-curvature}.
\end{proof}

\section{Finite index and total curvature}

We use the $L^1$ Michael--Simon inequality
\cite[Theorem~2.1, p.~368]{MS}. On a smooth minimal immersion
without boundary,
its form for a nonnegative function $h\in C_c^1(M)$ is
\begin{equation}\label{eq:ms-l1}
 \left(\int_M h^{n/(n-1)}\dd\right)^{(n-1)/n}
 \leq I_n\int_M|\nabla h|\dd.
\end{equation}
We explain the application to abstract immersed domains, since
the source states its theorem for ambient test functions.
For $\xi\in\R^{n+1}$, put
\[
 S_t(\xi)=X^{-1}(B_t^{\R^{n+1}}(\xi)),\qquad
 P_\xi(t)=\int_{S_t(\xi)}h\dd,\qquad
 E_\xi(t)=\int_{S_t(\xi)}|\nabla h|\dd.
\]
These integrals are finite because the integrands have compact
abstract support. Integrating the divergence of
$h\lambda(t-|X-\xi|)(X-\xi)^T$ on $M$, with the smooth radial
cutoffs of \cite[Lemma~2.2]{MS}, uses only
$\diver_M(X-\xi)^T=n$ and $|\nabla|X-\xi||\leq1$.
It gives the same radial inequality as in that lemma; after
approximating the ball indicator, it reads in distributions
\[
 -\frac{d}{dt}\bigl(t^{-n}P_\xi(t)\bigr)
 \leq t^{-n}E_\xi(t).
\]
A local immersion chart at each $p\in M$ gives
$\liminf_{t\downarrow0}t^{-n}P_{X(p)}(t)\geq\omega_n h(p)$,
where $\omega_n$ is the volume of the unit $n$-ball.
Thus Lemma~2.3 and the covering argument in the proof of
\cite[Theorem~2.1]{MS} apply to $X(\{h\geq1\})$:
the enlarged balls cover all abstract points of $\{h\geq1\}$,
and disjoint ambient balls have disjoint full preimages.
The ensuing superlevel-set integration proves \eqref{eq:ms-l1}
with the same dimension-dependent constant. Every integral counts
the abstract sheets, including coincident sheets; neither
embeddedness nor properness is required.

For $n\geq3$, applying \eqref{eq:ms-l1} to
$h=|f|^{2(n-1)/(n-2)}$ and using Cauchy--Schwarz gives
\begin{equation}\label{eq:sobolev}
 \left(\int_M|f|^{2n/(n-2)}\dd\right)^{(n-2)/n}
 \leq S_n\int_M|\nabla f|^2\dd,
 \qquad f\in\Cc(M).
\end{equation}
For any smooth relatively compact domain $D\subset M$,
approximate $\chi_D$ by nonnegative functions supported in a
fixed compact tubular enlargement of $\overline D$.
Their gradient integrals tend to the induced boundary area,
so \eqref{eq:ms-l1} gives
\begin{equation}\label{eq:isoperimetric}
 \Vol(D)^{(n-1)/n}\leq I_n\Vol_{n-1}(\partial D).
\end{equation}
Both the domain volume and boundary area are measured on $M$.
We also use the eigenvalue bound following from
\eqref{eq:sobolev}:
\begin{equation}\label{eq:clr}
 \ind\left(\int_M|\nabla f|^2\dd-\int_M Vf^2\dd\right)
 \leq C_n\int_M V^{n/2}\dd
 \qquad(0\leq V\in L^{n/2}(M)).
\end{equation}
This is Frank--Lieb--Seiringer \cite[Theorem~2.1]{FLS} applied
to the closed nonnegative form $t[f]=\int_M|\nabla f|^2\dd$
on $H^1(M)\subset L^2(M,d\mu)$, with their parameters
$\kappa=n/2>1$, $q=2n/(n-2)$, $\omega\equiv1$, and
$S=S_n^{-1}$. The induced measure is sigma-finite.
Completeness gives the form core $C_c^\infty(M)$;
splitting real and imaginary parts, taking absolute values,
and truncating nonnegative functions at one verify their
Assumption~2.1. Equation~\eqref{eq:sobolev} extends to the
form domain by closure. Truncation of $V$ and that inequality
make $V\in L^{n/2}$ infinitesimally form bounded.
The min--max principle and form-core approximation identify
the negative eigenvalue count, with multiplicity, with the
compact-test index in \eqref{eq:clr}.

\begin{lemma}\label{lem:cohomology}
Let $n\geq3$. Suppose that a complete, connected, two-sided minimal
immersion $M^n\to\R^{n+1}$ without boundary is $\delta$-stable
outside a compact set and
$\delta>((n-1)/n)^2$. Then
\[
 \dim\mathcal H_{(2)}^1(M)<\infty,
 \qquad \dim H_c^1(M;\R)<\infty.
\]
Here $\mathcal H_{(2)}^1(M)$ denotes the space of
square-integrable harmonic one-forms.
\end{lemma}

\begin{proof}
The references to \cite{Fu,LW02} in the introduction describe
antecedents of the argument. We prove the localized estimate
and the compactly supported cohomology conclusion here.
Let $\omega\in\mathcal H_{(2)}^1(M)$ and $f=|\omega|$.
The trace-free property of the shape operator gives
$\Ric_M\geq-\frac{n-1}{n}|A|^2g$. Bochner's formula
and the refined Kato inequality yield, in the weak sense,
\begin{equation}\label{eq:bochner}
 f\Delta f\geq\frac1{n-1}|\nabla f|^2
                    -\frac{n-1}{n}|A|^2f^2.
\end{equation}
For a compactly supported cutoff $\eta$ in the stable
exterior, write
\[
 \begin{aligned}
 X_\eta&=\int\eta^2|\nabla f|^2,&
 Y_\eta&=\int f^2|\nabla\eta|^2,\\
 Z_\eta&=\int\eta f\langle\nabla\eta,\nabla f\rangle,&
 P_\eta&=\int|A|^2\eta^2f^2 .
 \end{aligned}
\]
Integration of \eqref{eq:bochner} and stability give
\[
 \frac n{n-1}X_\eta\leq\frac{n-1}{n}P_\eta-2Z_\eta,
 \qquad
 \delta P_\eta\leq X_\eta+2Z_\eta+Y_\eta.
\]
Consequently,
\begin{equation}\label{eq:coercive}
 \left(\frac n{n-1}-\frac{n-1}{n\delta}\right)X_\eta
 \leq
 2\left(\frac{n-1}{n\delta}-1\right)Z_\eta
 +\frac{n-1}{n\delta}Y_\eta.
\end{equation}
The coefficient on the left is positive precisely when
$\delta>((n-1)/n)^2$.

Since $|Z_\eta|\leq(X_\eta Y_\eta)^{1/2}$, Young's
inequality gives
\begin{equation}\label{eq:form-caccioppoli}
 \int|\nabla(\eta f)|^2
 \leq C(n,\delta)\int f^2|\nabla\eta|^2.
\end{equation}
The calculation is justified at the zeros of $\omega$
by replacing $|\omega|$ with
$(|\omega|^2+\varepsilon^2)^{1/2}$ on compact supports
and then letting $\varepsilon\downarrow0$.

Choose a fixed cutoff $\theta$ vanishing near the unstable
compact set and equal to one outside a larger compact domain.
Let $K_0\Subset\operatorname{int}K_1$ be fixed compact smooth
domains containing $\supp(1-\theta)$ in $\operatorname{int}K_0$.
Apply \eqref{eq:form-caccioppoli} and \eqref{eq:sobolev} with
$\eta=\theta\chi_R$, where $\chi_R=1$ on $B_R(p)$,
$\supp\chi_R\subset B_{2R}(p)$, and $|\nabla\chi_R|\leq2/R$.
For $q=2n/(n-2)$, this gives
\[
\|\theta\chi_R f\|_{L^q(M)}^2
\leq C\int_{K_0}f^2|\nabla\theta|^2\dd
+\frac{C}{R^2}\|\omega\|_{L^2(M)}^2.
\]
Since $\omega\in L^2(M)$, Fatou's lemma yields
$\|\theta f\|_{L^q(M)}\leq C\|\omega\|_{L^2(K_0)}$.
Interior estimates for the Hodge equation on $K_1$ give
$\|\omega\|_{L^q(K_0)}\leq C\|\omega\|_{L^2(K_1)}$.
Combining the two bounds, we obtain
\begin{equation}\label{eq:local-controls-global}
\|\omega\|_{L^q(M)}
\leq C\|\omega\|_{L^2(K_1)}.\end{equation}
The domains and constants are independent of $\omega$.

If $\mathcal H_{(2)}^1(M)$ were infinite dimensional,
unique continuation would allow a sequence orthonormal
for the $L^2(K_1)$ inner product.
Equation \eqref{eq:local-controls-global} bounds its
$L^2$ norms on any fixed larger compact domain.
Interior elliptic estimates and Rellich compactness
would then give a subsequence converging in $L^2(K_1)$,
a contradiction.

We give a direct injection of $H_c^1(M;\R)$ into this
finite-dimensional space. Let $\mathcal D^{1,2}(M)$
be the completion of $\Cc(M)$ for the Dirichlet norm.
By \eqref{eq:sobolev} it embeds in $L^q(M)$.
For a smooth compactly supported closed one-form
$\alpha$, Hilbert-space projection gives
$u\in\mathcal D^{1,2}(M)$ satisfying
\[
 \int_M\langle du,dv\rangle\dd
 =\int_M\langle\alpha,dv\rangle\dd
 \qquad(v\in\mathcal D^{1,2}(M)).
\]
The form $\alpha-du$ is closed, coclosed, and in $L^2$,
hence smooth and harmonic. If $\alpha$ is replaced by $\alpha+dv$ with $v\in\Cc(M)$,
then $u$ is replaced by $u+v$. Thus this linear map depends
only on the class of $\alpha$ in $H_c^1(M;\R)$.

Suppose $\alpha-du=0$. Outside a compact smooth domain
containing $\supp\alpha$, the function $u$ is constant
on each complementary component. Every component
with noncompact closure has infinite volume: it contains infinitely
many disjoint intrinsic unit balls away from the
boundary, and the minimal immersion gives a uniform
positive lower bound for their volumes.
Since $u\in L^q$, the constants on these components
vanish. There are only finitely many relatively compact
components, because the boundary of the domain is a
compact smooth manifold. Filling those components
shows that $u$ has compact support. Elliptic regularity
makes $u$ smooth, so $[\alpha]=0$ in $H_c^1(M;\R)$.
The injection is proved.
\end{proof}

\begin{lemma}\label{lem:homology}
Let $n\geq1$ and let $M^n$ be a connected, oriented, noncompact
manifold without boundary with $\dim H_c^1(M;\R)=b<\infty$.
Let $\Omega\Subset M$
be a connected smooth domain such that every component
of $M\setminus\overline\Omega$ has noncompact closure.
Then $\partial\Omega$ has at most $b+1$ connected
components.
\end{lemma}

\begin{proof}
Write $\partial\Omega=\Sigma_1\cup\cdots\cup\Sigma_N$,
with the boundary orientations. The kernel of
\[
 H_{n-1}(\partial\Omega;\R)\longrightarrow
 H_{n-1}(M;\R)
\]
is one dimensional, generated by
$[\Sigma_1]+\cdots+[\Sigma_N]$.
To see this, triangulate $M$ compatibly with
$\partial\Omega$. A compact top-dimensional chain
whose boundary lies in $\partial\Omega$ has constant
multiplicity in each component of
$M\setminus\partial\Omega$. Its multiplicity is zero
in every component with noncompact closure, since the chain is
compactly supported, and is a single constant in the
connected domain $\Omega$. Its boundary is therefore
a multiple of the oriented sum above.
It follows that $N-1\leq\dim H_{n-1}(M;\R)$.
Poincar\'e duality identifies the latter dimension
with $\dim H_c^1(M;\R)=b$.
\end{proof}

\begin{proposition}\label{prop:volume}
Let $n\in\{3,4\}$ and $\delta>(n-2)/n$. Suppose that
$X:M^n\to\R^{n+1}$ is a complete, connected, two-sided
minimal immersion without boundary which is
$\delta$-stable outside a compact set. If
$\dim H_c^1(M;\R)<\infty$, then for every $p\in M$
there is $C<\infty$ such that
\[
 \Vol(B_R(p))\leq CR^n\qquad(R\geq1).
\]
In particular, $X$ is proper and has extrinsic
Euclidean volume growth.
\end{proposition}

\begin{proof}
We construct compact domains containing $B_R(p)$ whose boundaries
have area at most $CR^{n-1}$. We first recall the conformal
calculation and the area bound for each separator component from
\cite[Section~5]{HLW}. We then construct the separators in a
compact band in the stable exterior and use
Lemma~\ref{lem:homology} to bound their number.

Choose $z\in\R^{n+1}\setminus X(M)$, and put
\[
 r=|X-z|>0,\qquad \tg=r^{-2}g,\qquad
 s=|\nabla r|_g^2\leq1.
\]
Such a point $z$ exists since the smooth immersed
image has zero $(n+1)$-dimensional Lebesgue measure.
For a compactly supported function $\psi$ in the
stable exterior, direct calculation gives
\begin{equation}\label{eq:conformal-form}
 Q_\delta\bigl(r^{-(n-2)/2}\psi\bigr)
 =\int_M\bigl(|\tgrad\psi|^2-V\psi^2\bigr)\td,
 \qquad
 V=\delta r^2|A|^2-\frac{n(n-2)}2+\frac{n^2-4}{4}s.
\end{equation}
For example,
$\tDelta\log r=n(1-s)$; expanding the Dirichlet
term and integrating its cross term gives
\eqref{eq:conformal-form}. Compare \cite[Proposition~5.3]{HLW}.

The positive-supersolution characterization of
nonnegative Schr\"odinger forms, applied by exhaustion
(cf.\ \cite[Proposition~2.5]{ChNotes}), supplies on every
component of the stable exterior, a smooth function
$u>0$ such that
\begin{equation}\label{eq:positive}
 -\tDelta u\geq Vu.
\end{equation}
Only this exterior function will be used.

We first estimate a closed component of a weighted separator;
the compact band and the minimizer will be constructed below.
All quantities in this calculation use $\tg$.
In the band, consider the functional
\begin{equation}\label{eq:bubble-functional}
 \mathcal A_k(E)
 =\int_{\partial^*E}u^k\,d\widetilde\sigma
  -\int(\chi_E-\chi_{E_0})hu^k\,d\widetilde\mu,
\end{equation}
where the perimeter and volume terms are taken relative to the band.
The competitors are sets of finite perimeter fixed near its two
edges, and $E_0$ has the same prescribed behavior.
Let $\Sigma$ be a closed smooth component of the
free boundary of a minimizer. Denote its outer
normal by $\nu$, set $B(Y,Z)=\tg(\widetilde\nabla_Y\nu,Z)$,
and let $H=\operatorname{tr}B$.
With $v=\log u$, the first and second variations
are
\begin{equation}\label{eq:first-variation}
 H=h-k\nu(v)
\end{equation}
and
\begin{equation}\label{eq:second-variation}
 0\leq\int_\Sigma u^k\left\{
 |\nabla_\Sigma\psi|^2-
 \bigl(|B|^2+\widetilde\Ric(\nu,\nu)
       -k\widetilde{\operatorname{Hess}}v(\nu,\nu)
       +\nu(h)\bigr)\psi^2\right\}.
\end{equation}
These are local variational identities; compare
\cite[Lemma~5.2]{HLW}.

Substitute $\psi=u^{-k/2}\phi$ in
\eqref{eq:second-variation}, use
\eqref{eq:first-variation}, and integrate the
tangential Laplacian of $v$. For $0<k\leq2$ this
gives
\begin{equation}\label{eq:separator-inequality}
 \begin{split}
 \frac4{4-k}\int_\Sigma|\nabla_\Sigma\phi|^2
 \geq\int_\Sigma\biggl(
  -k\frac{\tDelta u}{u}+|B|^2
  +\widetilde\Ric(\nu,\nu)-\frac{H^2}{2}
  +\frac{h^2}{2}+\nu(h)\biggr)\phi^2 .
 \end{split}
\end{equation}
Indeed, the tangential expression before its last
estimate is
\[
 |\nabla\phi|^2+k\phi\langle\nabla\phi,\nabla v\rangle
 -\frac{k(4-k)}4\phi^2|\nabla v|^2
 \leq\frac4{4-k}|\nabla\phi|^2,
\]
and the discarded normal-gradient term has
coefficient $k(2-k)/2\geq0$.

When $n=3$, take $k=2$ and $\kappa=1/2$.
The scalar curvature of $\tg$ is
\[
 \widetilde\Scal=-r^2|A|^2+12-10s.
\]
Together with \eqref{eq:positive}, this gives
\begin{equation}\label{eq:scalar-bound}
 -2\frac{\tDelta u}{u}+\frac{\widetilde\Scal}{2}
 \geq(2\delta-\tfrac12)r^2|A|^2+3-\tfrac52s
 \geq\frac12.
\end{equation}
If $h$ satisfies
\begin{equation}\label{eq:h-barrier}
 2|\tgrad h|\leq\kappa+h^2,
\end{equation}
the Gauss equation in \eqref{eq:separator-inequality}
yields
\begin{equation}\label{eq:surface-separator}
 2\int_\Sigma|\nabla_\Sigma\phi|^2
 \geq\int_\Sigma(\tfrac14-K_\Sigma)\phi^2 .
\end{equation}
The constant test function and Gauss--Bonnet imply
\begin{equation}\label{eq:area-three}
 \Vol_{2,\tg}(\Sigma)\leq16\pi.
\end{equation}
The surface is oriented because it is a boundary
in an oriented three-manifold.

When $n=4$, take $k=5/3$ and $\kappa=1/3$.
For orthonormal $e,\nu$, define
\[
 \widetilde\biRic(e,\nu)
 =\widetilde\Ric(e,e)+\widetilde\Ric(\nu,\nu)
  -\widetilde{\operatorname{Sec}}(e,\nu).
\]
The conformal estimate in \cite[Proposition~5.5]{HLW}
is
\begin{equation}\label{eq:biric}
 \widetilde\biRic(e,\nu)
 \geq7-5s-\frac56r^2|A|^2.
\end{equation}
The hypersurface Gauss equation implies
\[
 |B|^2+\widetilde\Ric(\nu,\nu)-\frac{H^2}{2}
      +\Ric_\Sigma(e,e)
 \geq\widetilde\biRic(e,\nu).
\]
In an orthonormal frame with
$e=e_1$, the residual quadratic expression in $B$ is
\[
 \frac12B_{11}^2+\frac12(B_{22}-B_{33})^2
 +B_{12}^2+B_{13}^2+2B_{23}^2\geq0.
\]
Equations \eqref{eq:positive} and \eqref{eq:biric}
give
\[
 -\frac53\frac{\tDelta u}{u}+
 \widetilde\biRic(e,\nu)
 \geq\frac53(\delta-\tfrac12)r^2|A|^2+\frac13
 \geq\frac13.
\]
Let $\lambda_\Sigma(x)$ be the least eigenvalue of
$\Ric_\Sigma$ at $x\in\Sigma$. Under \eqref{eq:h-barrier},
\eqref{eq:separator-inequality} becomes
\begin{equation}\label{eq:spectral-ricci}
 \frac{12}{7}\int_\Sigma|\nabla_\Sigma\phi|^2
 \geq\int_\Sigma(\tfrac16-\lambda_\Sigma)\phi^2 .
\end{equation}
Since $12/7<2$, \eqref{eq:spectral-ricci} implies
$\lambda_1(-2\Delta_\Sigma+\lambda_\Sigma)\geq1/6$.

Rescale the induced metric on $\Sigma$ by $1/12$.
The corresponding spectral lower bound is then two.
Each component $\Sigma$ is a compact smooth three-manifold
without boundary. Apply Antonelli--Xu
\cite[Theorem~1(2) and Remark~1]{AX} with their dimension
equal to three, $\gamma=2=(3-1)/(3-2)$, and $\lambda=1$.
Remark~1 permits the quadratic-form hypothesis when the
least Ricci eigenvalue is only continuous. The theorem bounds
the volume of the universal cover by $|\Sph^3|$, hence also
the volume of $\Sigma$ in this rescaled metric.
Returning to the original conformal metric yields,
as in \cite[Theorem~5.6]{HLW},
\begin{equation}\label{eq:area-four}
 \Vol_{3,\tg}(\Sigma)\leq24\sqrt3\,|\Sph^3|.
\end{equation}
Thus, in either dimension, every closed connected
free-boundary component has uniformly bounded
$(n-1)$-volume in $\tg$. These componentwise estimates also apply when the band has
disconnected outer boundary.

We now place the construction inside the stable
exterior. Fix $p\in M$, set
$\rho=d_M(p,\cdot)$ and $C_0=r(p)+1$, and note that
\begin{equation}\label{eq:log-control}
 r\leq\rho+C_0,\qquad
 |d\log r|_{\tg}\leq1,\qquad
 |d\log(\rho+C_0)|_{\tg}\leq1
\end{equation}
almost everywhere. Choose $L>4\pi/\sqrt\kappa+10$ and
$\Lambda>2e^L$.
For large $R$, take a connected compact smooth
domain satisfying
\[
 B_{\Lambda R}(p)\subset D_R\subset B_{\Lambda R+1}(p).
\]
Let $d_0$ be the distance to the whole of $\partial D_R$
measured by curves in $(D_R,\tg)$.
For $q\in D_R$ with $d_0(q)<L$, integration of \eqref{eq:log-control} along
a curve from $\partial D_R$ gives
\begin{equation}\label{eq:collar-control}
 \rho(q)+C_0\geq e^{-L}(\Lambda R+C_0),
 \qquad
 r(q)\leq e^L(\Lambda R+1+C_0).
\end{equation}
For large $R$, this band is disjoint from
$B_R(p)$ and lies wholly in the stable exterior.
The ball $B_R(p)$ is beyond the inner edge.
The closure of the band is compact in the
abstract manifold; no properness has been used.

Choose a smooth approximation $d$ of $d_0$ in $\{0<d_0<L\}$,
with arbitrarily small uniform error and $|\tgrad d|\leq2$.
Choose regular values $a<b$ in the collar such that
$b-a>4\pi/\sqrt\kappa$, and on $a<d<b$ set
\begin{equation}\label{eq:cot-barrier}
 h(d)=\sqrt\kappa\cot\frac{\pi(d-a)}{b-a}.
\end{equation}
Then
\[
 2|\tgrad h|
 \leq\frac{4\pi\sqrt\kappa}{b-a}
          \csc^2\frac{\pi(d-a)}{b-a}
 \leq\kappa+h^2.
\]
The function $h$ tends to $+\infty$ at the outer
edge and to $-\infty$ at the inner edge.

In \eqref{eq:bubble-functional}, require $E$ to
contain a neighborhood of the outer edge and to
avoid a neighborhood of the inner edge.
The relative existence and regularity input is
\cite[Proposition~6.12 and its proof]{ChNotes}.
To reduce the weighted problem to its unweighted functional,
use, on the compact band $\{a\leq d\leq b\}$,
\[
 \bar g=u^{2k/(n-1)}\tg,\qquad
 \bar h=h\,u^{-k/(n-1)}.
\]
Indeed, $d\bar\sigma=u^k d\widetilde\sigma$ and
$\bar h\,d\bar\mu=h u^k d\widetilde\mu$, so the two
relative functionals agree. Since $u$ is smooth and positive
on this compact band, $\bar g$ is smooth up to its boundary
and $\bar h$ has the prescribed opposite divergences.
The boundary $d=a$ is the positive edge, where the set is full,
and $d=b$ is the negative edge, where it is empty.
Since $a$ and $b$ are regular values, the level hypersurfaces
near these two levels have uniformly bounded weighted mean
curvature $H+k\nu(\log u)$. The divergence of $h$ at the edges
therefore gives strict barriers at $a+\varepsilon_0$ and
$b-\varepsilon_0$, for some $\varepsilon_0>0$.
These barrier bounds and $\varepsilon_0$ may depend on the
fixed band and on $R$; uniformity as $R\to\infty$ is not needed.

Minimize first on $\{a+\varepsilon<d<b-\varepsilon\}$, where
$0<\varepsilon<\varepsilon_0$, extending competitors by the
prescribed full and empty sets across its boundary.
Here $u^k$ is smooth and bounded above and below by positive
constants, so the direct method gives a minimizer.

Comparison with the level hypersurfaces confines its free boundary
to $\{a+\varepsilon_0<d<b-\varepsilon_0\}$, independently of
$0<\varepsilon<\varepsilon_0$.
The difference $\chi_E-\chi_{E_0}$ vanishes near the singular
edges, so the volume term is finite.
The minimizer is unconstrained near its free boundary, which is
smooth because the band has dimension $n=3$ or $4$; this is
the interior regularity input in \cite[Proposition~6.12]{ChNotes}
applied to $\bar g$ and the smooth interior forcing $\bar h$.
The first and second variations therefore apply.
Here \cite[Section~5.1]{HLW} supplies the weighted variational
formulas used above; existence is supplied by the relative
construction just described.
The minimization is performed on every component of the band,
with the prescribed condition on each boundary level that is present.
If a component has only one type of edge, the same truncated
direct method and its barrier at that edge apply. This uses
the proof of the relative existence result and does not assume
that each component has both edge types.

Let $E_R$ be the extension of the minimizing set by the prescribed
full outer region up to $\partial D_R$ and by the prescribed empty
inner region. Let $U_R$ be the connected component of
$D_R\setminus\overline{E_R}$ containing $B_R(p)$.
Its boundary is a union of closed separator components.
Let $\mathcal F_R$ be the collection of components of
$M\setminus\overline{U_R}$ with compact closure, and set
\[
\Omega_R=\operatorname{int}\left(
\overline{U_R}\cup\bigcup_{F\in\mathcal F_R}\overline F\right).
\]
The collection $\mathcal F_R$ is finite because $\partial U_R$
is compact and smooth. Thus $\Omega_R$ is a connected relatively
compact smooth domain containing $B_R(p)$, and every component of
$M\setminus\overline{\Omega_R}$ has noncompact closure.
Filling removes boundary components and creates no new boundary,
so $\partial\Omega_R$ is a subcollection of the separators.

Put $b_1^c=\dim H_c^1(M;\R)$.
Lemma~\ref{lem:homology} shows that
$\partial\Omega_R$ has at most $b_1^c+1$
components. On each of them, \eqref{eq:collar-control}
gives $r\leq CR$, and \eqref{eq:area-three} or
\eqref{eq:area-four} gives a uniform conformal
area bound. Consequently,
\[
 \begin{split}
 \Vol_{n-1,g}(\partial\Omega_R)
 &=\sum_{\Sigma\subset\partial\Omega_R}
      \int_\Sigma r^{n-1}\,d\widetilde\sigma\\
 &\leq C(b_1^c+1)R^{n-1}.
 \end{split}
\]
Apply \eqref{eq:isoperimetric} to the relatively compact
abstract domain $\Omega_R$ with the original induced metric $g$.
This gives
\[
 \Vol(B_R(p))\leq\Vol(\Omega_R)
 \leq C\Vol_{n-1}(\partial\Omega_R)^{n/(n-1)}
 \leq C'R^n.
\]
Enlarging $C'$ covers bounded radii.
Finally, apply \cite[Theorem~1.1]{FS} to the
complete original immersion. It gives properness
and extrinsic Euclidean volume growth.
\end{proof}

\begin{lemma}\label{lem:decay}
Let $n\in\{3,4\}$ and $\delta>(n-2)/n$.
Suppose that $X:M^n\to\R^{n+1}$ is a complete two-sided minimal
immersion without boundary which is proper, has extrinsic
Euclidean volume growth, and is $\delta$-stable outside a
compact set. Then
$|X|\,|A|\to0$ at infinity, and $\ind_1(M)<\infty$.
\end{lemma}

\begin{proof}
Decrease the coefficient so that
$(n-2)/n<\delta<1$.
Choose a nonempty compact domain $K$ outside which
the immersion is $\delta$-stable, and take
$z\notin X(M)$ and $r=|X-z|$ as above.
Proposition~\ref{prop:curvature} and
$d_M(x,K)\geq r(x)-\sup_Kr$ imply
\begin{equation}\label{eq:scale-curvature}
 |A|(x)\leq\frac{C}{d_M(x,K)},
 \qquad r(x)|A|(x)\leq C'
\end{equation}
outside a fixed extrinsic ball.

Suppose there were points $x_j\to\infty$ with
$r(x_j)|A|(x_j)\geq\varepsilon_0>0$.
By properness, $R_j=r(x_j)\to\infty$.
Set $\Sigma_j=M$, $F_j=R_j^{-1}(X-z)$, and $g_j=R_j^{-2}g$.
The marked points satisfy
\begin{equation}\label{eq:blowdown-marking}
 |F_j(x_j)|=1,\qquad
 |A_j|(x_j)=R_j|A|(x_j)\geq\varepsilon_0.
\end{equation}
Write $\mu_j=(F_j)_*(d\mu_{g_j})$; explicitly, for every
ambient Borel set $E$,
\[
 \mu_j(E)=\Vol_{g_j}(F_j^{-1}(E)).
\]
Thus mass is counted on the full abstract immersion, with
multiplicity even when different sheets have the same image.
The extrinsic volume bound gives a subsequential stationary
integral varifold limit $C$. By monotonicity, the finite limit
\[
\Theta_\infty=\lim_{R\to\infty}
R^{-n}\Vol\bigl(X^{-1}(B_R^{\R^{n+1}}(z))\bigr)
\]
exists. At every continuity radius $a>0$ of the limiting mass,
the rescaled volume ratio tends to $\Theta_\infty$.
Thus the volume ratio of $C$ is constant, and equality in the
monotonicity formula shows that $C$ is a cone.
On every compact
annulus about the origin, \eqref{eq:scale-curvature}
gives a uniform bound for the rescaled second
fundamental forms.

Set $\mathcal O=\R^{n+1}\setminus\{0\}$ and
\[
    A_{a,b}:=\{x\in\mathbb R^{n+1}:a<|x|<b\}.
\]
We construct one abstract immersed limit over $\mathcal O$.

Fix
\[
    0<a<a'<b'<b<\infty.
\]
The curvature estimate on $A_{a,b}$ gives a uniform graphical
radius for all sufficiently large $j$. Decrease it using $a',b'$
to choose
\[
 \rho=\rho(n,\delta,C',a,a',b',b)>0,\qquad
 2\rho<\min\{a'-a,b-b'\}.
\]
Here $C'$ is the fixed bound in \eqref{eq:scale-curvature}.
The radius and a constant $c=c(n)>0$ are independent of $j$.
For every
\(p\in F_j^{-1}(\overline A_{a',b'})\),
the intrinsic ball \(B^{g_j}_{2\rho}(p)\) is contained in
\(F_j^{-1}(A_{a,b})\), is represented by a graph with uniform
estimates, and
\[
    \operatorname{vol}_{g_j} B^{g_j}_{\rho}(p)
        \ge c\rho^n .
\]
Choose a maximal \(2\rho\)-separated family
\(\{p_{j,\alpha}\}_{\alpha=1}^{N_j}\) in
\(F_j^{-1}(\overline A_{a',b'})\).  The balls
\(B^{g_j}_{\rho}(p_{j,\alpha})\) are disjoint in the abstract
domain, irrespective of intersections of their images, and hence
\[
    N_j c\rho^n
    \le \mu_j(A_{a,b}).
\]
The local mass bound therefore gives a uniform bound for \(N_j\).
By maximality, the balls \(B^{g_j}_{2\rho}(p_{j,\alpha})\) cover
\(F_j^{-1}(\overline A_{a',b'})\). Since the unrestricted
rescaled metrics $g_j$ are complete, their closed intrinsic
balls are compact; consequently
\(F_j^{-1}(\overline A_{a',b'})\) is compact.  Thus, after deleting
the preimage of the origin,
\[
    F_j:F_j^{-1}(\mathbb R^{n+1}\setminus\{0\})
        \longrightarrow \mathbb R^{n+1}\setminus\{0\}
\]
is a proper immersion. This also follows directly, for every $j$,
from the assumed properness of $X$.

The graphical estimates and the interior estimates for the minimal
surface equation give uniform bounds for all covariant derivatives
of the second fundamental form on compact subsets of
\(\mathcal O\). For each compact subset and derivative order,
the bound is independent of $j$; finitely many initial $j$ can
be included by smoothness and properness of the approximants.
The mass bounds hold on these same compact subsets, and
\eqref{eq:blowdown-marking} shows that every image meets the
fixed compact set $\Sph^n\subset\mathcal O$.
Breuning \cite[Corollary~7.13]{Breuning}, applied to this open
target, therefore produces, after passing to a subsequence,
an abstract manifold \(\Sigma_\infty\) without boundary and a
proper smooth immersion
\[
    F_\infty:\Sigma_\infty
        \longrightarrow\mathbb R^{n+1}\setminus\{0\}.
\]
More precisely, that corollary supplies open exhaustions
\[
 U_j\Subset U_{j+1},\quad \bigcup_jU_j=\Sigma_\infty,
 \qquad
 \mathcal O_j\Subset\mathcal O_{j+1},\quad
 \bigcup_j\mathcal O_j=\mathcal O,
\]
and diffeomorphisms onto the entire indicated preimages,
\begin{equation}\label{eq:immersion-exhaustion}
 \varphi_j:U_j\longrightarrow F_j^{-1}(\mathcal O_j),
 \qquad
 \|F_j\circ\varphi_j-F_\infty\|_{C^0(U_j)}\longrightarrow0,
\end{equation}
with $F_j\circ\varphi_j\to F_\infty$ locally smoothly on
$\Sigma_\infty$. In particular the limit is minimal.
The theorem constructs the abstract manifold and these
parametrizations on one exhaustion, providing compatible charts
across overlapping annuli and no artificial annular boundary.
Properness of $F_\infty$ is also a conclusion of the theorem;
thus $F_\infty^{-1}(\overline A_{a',b'})$ is compact.

To check coverage and mass identification, fix compact target
annuli $K_*,K_*^+$ with
$K_*\subset\operatorname{int}K_*^+\Subset\mathcal O$.
For large $j$, $K_*\subset\mathcal O_j$.
Every $y\in F_j^{-1}(K_*)$ has a preimage
$p=\varphi_j^{-1}(y)$, and the uniform $C^0$ estimate in
\eqref{eq:immersion-exhaustion} places $F_\infty(p)$ in $K_*^+$.
Hence all such $p$ lie in the fixed compact set
$F_\infty^{-1}(K_*^+)$, contained in $U_j$ for large $j$.
Change of variables and local smooth convergence on this compact
set give convergence of the full induced-volume integrals for
ambient tests supported in $K_*$. The same argument applies
to tests depending on tangent planes, identifying the induced
varifold of $F_\infty$ with $C$ on $\mathcal O$.
Transverse crossing branches remain distinct abstract sheets;
coincident sheets contribute separately to these integrals.
No sheets over a compact subannulus are discarded, and no
embeddedness of the support is asserted.

In particular, the points $p_j=\varphi_j^{-1}(x_j)$ lie in
a fixed compact annular preimage. After taking a subsequence,
$p_j\to p_\infty\in\Sigma_\infty$, and smooth convergence
together with \eqref{eq:blowdown-marking} gives
\begin{equation}\label{eq:retained-mark}
 |F_\infty(p_\infty)|=1,\qquad
 |A_\infty|(p_\infty)\geq\varepsilon_0.
\end{equation}

The associated measure is the restriction of the limiting cone
varifold away from the origin.  Equality in the monotonicity formula
therefore gives
\[
    F_\infty^\perp=0
\]
on every sheet.  If \(r=|F_\infty|\), then
\[
    |\nabla_{\Sigma_\infty}r|=1,
    \qquad
    Y:=r\nabla r
      =\nabla_{\Sigma_\infty}\frac{r^2}{2},
    \qquad
    dF_\infty(Y)=F_\infty .
\]
Since \(F_\infty\) is proper, the flow \(\Phi_t\) of \(Y\) exists
for every \(t\in\mathbb R\); indeed a trajectory on a bounded time
interval remains in the preimage of a compact annulus.  Moreover,
\[
    F_\infty(\Phi_t(p))=e^tF_\infty(p).
\]
Hence
\[
    \Gamma:=r^{-1}(1)=F_\infty^{-1}(S^n)
\]
is a compact smooth manifold, possibly disconnected, and
\[
    \mathbb R\times\Gamma\longrightarrow\Sigma_\infty,
    \qquad
    (t,q)\longmapsto\Phi_t(q),
\]
is a diffeomorphism.  Equivalently,
\[
    \Sigma_\infty\simeq(0,\infty)\times\Gamma,
    \qquad
    F_\infty(s,q)=sF_\infty(1,q).
\]
All subsequent integrals are taken on this abstract immersion,
with multiplicity.

For $\eta>0$ and a smooth radial function $\phi(r)$ supported
in a compact annulus, the functions
\[
 f_j=(|A_j|^2+\eta)^{\delta/2}\phi(r)
\]
are compactly supported on the rescaled abstract
immersions by properness. Their supports lie
outside the rescaled unstable core for large
$j$. For fixed $\eta>0$, the full-preimage coverage just proved
places all these supports in a fixed compact part of the abstract
limit under $\varphi_j^{-1}$. Smooth convergence therefore passes
their stability inequalities, including all sheets, to the cone
as $j\to\infty$.

We use the local calculation in
\cite[Equations~(4.4)--(4.6)]{HLW} with $p=2-\delta$.
For $n=3,4$ and $(n-2)/n<\delta<1$, we have
$0<p\leq1+2/(n-1)$.
Combining the regularized Simons inequality with the preceding
test gives, after discarding nonnegative terms,
\[
\delta(1+\delta)\int_C r^{-2}|A_C|^2
(|A_C|^2+\eta)^{\delta-1}\phi^2
\leq\int_C (|A_C|^2+\eta)^\delta|\nabla\phi|^2.
\]
Since $0<\delta<1$, the left integrand is bounded by
$r^{-2}|A_C|^{2\delta}\phi^2$. The test is supported in a
compact annulus, so dominated convergence as $\eta\downarrow0$
gives
\begin{equation}\label{eq:cone-inequality}
\delta(1+\delta)\int_C r^{-2}|A_C|^{2\delta}\phi^2
\leq\int_C|A_C|^{2\delta}|\nabla\phi|^2.
\end{equation}
Only the annular tests defined above are used. The further inequality $\delta>(n-2)^2/[4(n-1)]$,
needed in the radial argument, follows from
$\delta>(n-2)/n$ in dimensions three and four.

The sharp radial Hardy quotient now excludes a nonflat cone.
If $A_C\not\equiv0$, separate the positive
angular factor using
$|A_C|(r,\vartheta)=r^{-1}|A_C|(1,\vartheta)$.
Put $\beta=n-2-2\delta$, and choose
$0\neq\chi\in\Cc(\R)$. For
\[
 \phi_T(r)=r^{-\beta/2}\chi(\log r/T),
\]
the quotient of the right radial integral in
\eqref{eq:cone-inequality} by the left integral
without its factor $\delta(1+\delta)$ is
\[
 \frac{\beta^2}{4}
 +\frac1{T^2}\frac{\int_\R|\chi'|^2}{\int_\R|\chi|^2}.
\]
Letting $T\to\infty$ would imply
\[
 \delta(1+\delta)\leq\frac{(n-2-2\delta)^2}{4},
 \quad\text{or equivalently}\quad
 \delta\leq\frac{(n-2)^2}{4(n-1)},
\]
a contradiction. Thus $A_C\equiv0$.
This contradicts the retained curvature bound
\eqref{eq:retained-mark} at $p_\infty$. We have proved
\begin{equation}\label{eq:small-curvature}
 r|A|\longrightarrow0.
\end{equation}
Translation of the center does not change this
conclusion at infinity.

The extrinsic Hardy inequality is
\begin{equation}\label{eq:hardy}
 \frac{(n-2)^2}{4}\int_M\frac{f^2}{r^2}\dd
 \leq\int_M|\nabla f|^2\dd,\qquad f\in\Cc(M).
\end{equation}
To verify it, let $Y=\nabla r/r$.
Minimality gives
$\diver Y=(n-2|\nabla r|^2)/r^2$.
Expand the nonnegative integral
$\int|\nabla f+\frac{n-2}{2}fY|^2$ and use
$|\nabla r|\leq1$.

By \eqref{eq:small-curvature}, outside a compact
set $K'$ we have
$|A|^2\leq(n-2)^2/(8r^2)$. Properness makes
$K'$ compact. It follows from \eqref{eq:hardy}
that
\[
 Q_1(f)\geq\frac12\int_M|\nabla f|^2\dd
                  -\int_{K'}|A|^2f^2\dd.
\]
The right-hand form has finite index by
\eqref{eq:clr}: multiply it by two and use
$V=2\mathbf1_{K'}|A|^2$, which is bounded and compactly
supported. Hence $\ind_1(M)<\infty$.
\end{proof}

\begin{proof}[Proof of Theorem~\ref{thm:index}]
Suppose first that $\ind_\delta(M)<\infty$.
There is a compact domain outside which
$Q_\delta$ is nonnegative. Otherwise one could
choose negative test functions successively
outside the supports of all previous ones.
Their disjoint supports would span negative
subspaces of arbitrarily large dimension.

Lemma~\ref{lem:cohomology} gives
$\dim H_c^1(M;\R)<\infty$.
Since
\[
 \left(\frac{n-1}{n}\right)^2
 =\frac{n-2}{n}+\frac1{n^2},
\]
Proposition~\ref{prop:volume} applies. It gives
intrinsic Euclidean volume growth, properness,
and \eqref{eq:extrinsic-volume}.
Lemma~\ref{lem:decay} now gives
\eqref{eq:curvature-decay} and finite ordinary
index.
For $n=3$, \cite[Theorem~5]{CL} implies
$\int_M|A|^3<\infty$; for $n=4$, use
\cite[Theorem~1.4]{CLMS}.

Conversely, if $\int_M|A|^n\dd<\infty$,
apply \eqref{eq:clr} with $V=\tau|A|^2$.
For every $\tau>0$,
\begin{equation}\label{eq:index-bound}
 \ind_\tau(M)\leq C_n\tau^{n/2}\int_M|A|^n\dd
 <\infty.
\end{equation}
In particular, the ordinary index is finite.
The citation to \cite{Tysk} in the introduction records the
classical provenance of this direction; the proof here is
the Sobolev--CLR argument above.
The implication from finite ordinary index to
finite total curvature is again supplied by
\cite{CL,CLMS}. This proves all equivalences
and the remaining assertions.
\end{proof}

\begin{proof}[Proof of Corollary~\ref{cor:spectrum}]
Theorem~\ref{thm:index} gives $|A|\to0$ at
infinity. Hence $|A|$ is bounded, and the
Schr\"odinger operators in the statement are
bounded below and self-adjoint on the domain
of $-\Delta$.
The Gauss equation gives
$\Ric(v,v)=-|Av|^2$, so the Ricci curvature
tends to zero at infinity. Theorem~1 of
Lu--Zhou \cite{LZ} yields
$\sigma_{\ess}(-\Delta)=[0,\infty)$.

Multiplication by $|A|^2$ is a relatively
compact perturbation of $-\Delta$.
Indeed, its restriction to a compact set
composed with $(-\Delta+1)^{-1}$ is compact
by local elliptic estimates and Rellich
compactness. The operator norm of the
remaining tail is bounded by the supremum
of $|A|^2$ on that tail, which tends to zero.
Weyl's theorem therefore gives
$\sigma_{\ess}(L_\tau)=[0,\infty)$ for every
$\tau\geq0$.
Finally, \eqref{eq:index-bound} bounds the
dimension of the negative spectral subspace
for each $\tau>0$, proving the last assertion.
\end{proof}


\begin{thebibliography}{99}

\bibitem{AX}
G. Antonelli and K. Xu,
\emph{New spectral Bishop--Gromov and Bonnet--Myers theorems
and applications to isoperimetry},
J. Eur. Math. Soc., to appear;
\href{https://arxiv.org/abs/2405.08918}{arXiv:2405.08918},
version 4 (2026).

\bibitem{Breuning}
P. Breuning,
\emph{Immersions with bounded second fundamental form},
preprint (2012),
\href{https://arxiv.org/abs/1201.4562}{arXiv:1201.4562},
version 1.

\bibitem{CCMMR}
X. Cabr\'e, G. Catino, L. Mari, P. Mastrolia, and A. Roncoroni,
\emph{Gradient estimates for the Green kernel under spectral Ricci
bounds, and the stable Bernstein theorem in $\R^4$},
preprint (2026),
\href{https://arxiv.org/abs/2604.14393}{arXiv:2604.14393}.

\bibitem{CMMR}
G. Catino, L. Mari, P. Mastrolia, and A. Roncoroni,
\emph{Criticality, splitting theorems under spectral Ricci
bounds and the topology of stable minimal hypersurfaces},
preprint, \href{https://arxiv.org/abs/2412.12631}{arXiv:2412.12631},
version 5 (2026).

\bibitem{CMR}
G. Catino, P. Mastrolia, and A. Roncoroni,
\emph{Two rigidity results for stable minimal hypersurfaces},
Geom. Funct. Anal. \textbf{34} (2024), no.~1, 1--18.
\href{https://doi.org/10.1007/s00039-024-00662-1}
{doi:10.1007/s00039-024-00662-1}.

\bibitem{CW}
Q.-M. Cheng and G. Wei,
\emph{Complete two-sided $\delta$-stable minimal hypersurfaces
in $\R^{n+1}$},
preprint (2025),
\href{https://arxiv.org/abs/2507.00342}{arXiv:2507.00342}.

\bibitem{CZ}
X. Cheng and D. Zhou,
\emph{Manifolds with weighted Poincar\'e inequality and
uniqueness of minimal hypersurfaces},
Comm. Anal. Geom. \textbf{17} (2009), no.~1, 139--154.
\href{https://doi.org/10.4310/CAG.2009.v17.n1.a6}
{doi:10.4310/CAG.2009.v17.n1.a6}.

\bibitem{ChNotes}
O. Chodosh,
\emph{Stable minimal surfaces and positive scalar curvature},
lecture notes, Stanford University, 2021.
\url{https://web.stanford.edu/~ochodosh/Math258-min-surf.pdf}.

\bibitem{CLA}
O. Chodosh and C. Li,
\emph{Stable anisotropic minimal hypersurfaces in $\R^4$},
Forum Math. Pi \textbf{11} (2023), paper no.~e3, 22~pp.
\href{https://doi.org/10.1017/fmp.2023.1}
{doi:10.1017/fmp.2023.1}.

\bibitem{CL}
O. Chodosh and C. Li,
\emph{Stable minimal hypersurfaces in $\R^4$},
Acta Math. \textbf{233} (2024), no.~1, 1--31.
\href{https://doi.org/10.4310/ACTA.2024.v233.n1.a1}
{doi:10.4310/ACTA.2024.v233.n1.a1}.

\bibitem{CLMS}
O. Chodosh, C. Li, P. Minter, and D. Stryker,
\emph{Stable minimal hypersurfaces in $\R^5$},
Ann. of Math. (2) \textbf{204} (2026), no.~1, 423--453.
\href{https://doi.org/10.4007/annals.2026.204.1.5}
{doi:10.4007/annals.2026.204.1.5}.

\bibitem{DLZ}
Y. Dong, H. Lin, and W. Zhang,
\emph{Spectral Bernstein theorems for submanifolds in Euclidean spaces},
Math. Ann. \textbf{395} (2026), article~65.
\href{https://doi.org/10.1007/s00208-026-03488-4}
{doi:10.1007/s00208-026-03488-4}.

\bibitem{FS}
E. Florit-Simon,
\emph{Equivalence of intrinsic and extrinsic area bounds
for minimal surfaces},
preprint (2026),
\href{https://arxiv.org/abs/2605.06468}{arXiv:2605.06468}.

\bibitem{FLS}
R. L. Frank, E. H. Lieb, and R. Seiringer,
\emph{Equivalence of Sobolev inequalities and Lieb--Thirring
inequalities},
in \emph{XVIth International Congress on Mathematical Physics},
World Scientific, Hackensack, NJ, 2010, pp.~523--535.
\href{https://arxiv.org/abs/0909.5449}{arXiv:0909.5449}.

\bibitem{Fu}
H.-P. Fu,
\emph{The structure of $\delta$-stable minimal hypersurfaces in $\R^{n+1}$},
Hokkaido Math. J. \textbf{40} (2011), no.~1, 103--110.
\href{https://doi.org/10.14492/hokmj/1300108401}
{doi:10.14492/hokmj/1300108401}.

\bibitem{HLW}
H. Hong, H. Li, and G. Wang,
\emph{On $\delta$-stable minimal hypersurfaces in $\R^{n+1}$},
preprint (2024),
\href{https://arxiv.org/abs/2407.03222}{arXiv:2407.03222}.

\bibitem{HW}
H. Hong and G. Wang,
\emph{Mixed radial volume comparison under spectral Ricci bounds},
preprint (2026),
\href{https://arxiv.org/abs/2609.02430}{arXiv:2609.02430}.

\bibitem{LW02}
P. Li and J. Wang,
\emph{Minimal hypersurfaces with finite index},
Math. Res. Lett. \textbf{9} (2002), no.~1, 95--103.
\href{https://doi.org/10.4310/MRL.2002.v9.n1.a7}
{doi:10.4310/MRL.2002.v9.n1.a7}.

\bibitem{LZ}
Z. Lu and D. Zhou,
\emph{On the essential spectrum of complete non-compact manifolds},
J. Funct. Anal. \textbf{260} (2011), no.~11, 3283--3298.
\href{https://doi.org/10.1016/j.jfa.2010.10.010}
{doi:10.1016/j.jfa.2010.10.010}.

\bibitem{Mazet}
L. Mazet,
\emph{Stable minimal hypersurfaces in $\R^6$},
preprint (2024),
\href{https://arxiv.org/abs/2405.14676}{arXiv:2405.14676}.

\bibitem{MS}
J. H. Michael and L. M. Simon,
\emph{Sobolev and mean-value inequalities on generalized
submanifolds of $\R^n$},
Comm. Pure Appl. Math. \textbf{26} (1973), 361--379.
\href{https://doi.org/10.1002/cpa.3160260305}
{doi:10.1002/cpa.3160260305}.

\bibitem{TZ}
L.-F. Tam and D. Zhou,
\emph{Stability properties for the higher dimensional
catenoid in $\R^{n+1}$},
Proc. Amer. Math. Soc. \textbf{137} (2009), no.~10, 3451--3461.
\href{https://arxiv.org/abs/0708.3310}{arXiv:0708.3310}.

\bibitem{Tysk}
J. Tysk,
\emph{Finiteness of index and total scalar curvature for
minimal hypersurfaces},
Proc. Amer. Math. Soc. \textbf{105} (1989), no.~2, 429--435.
\href{https://doi.org/10.1090/S0002-9939-1989-0946639-1}
{doi:10.1090/S0002-9939-1989-0946639-1}.

\end{thebibliography}
\end{document}